\documentclass[12pt]{amsart}
\usepackage{amssymb,amsmath,amsthm,amssymb,amsfonts,amscd}
\usepackage{graphicx}
\usepackage{bm}
\usepackage{color, enumerate}
\usepackage[all]{xy}
\usepackage{bigints}
\usepackage{enumitem}
\usepackage[english]{babel}
\usepackage{dsfont} 
\usepackage{hyperref}
\usepackage{verbatim}
\usepackage[mathscr]{euscript}

\usepackage{soul}

\numberwithin{equation}{section}

\usepackage{etoolbox}
\patchcmd{\thmhead}{(#3)}{#3}{}{}

\newtheorem{theorem}{Theorem}[section]
\newtheorem{lemma}[theorem]{Lemma}
\newtheorem{definition}[theorem]{Definition}
\newtheorem{corollary}[theorem]{Corollary}
\newtheorem{proposition}[theorem]{Proposition}
\newtheorem{remark}[theorem]{Remark}

\DeclareMathAlphabet{\mathpzc}{OT1}{pzc}{m}{it}

\def\0{{\rm \bf{0}}}

\def\B{{\mathcal{B}}}

\DeclareMathOperator*{\esssup}{ess\,sup}
\DeclareMathOperator*{\essinf}{ess\,inf}

\def\B2star{\overline{B}_X^{w(X^{\ast\ast},X^{\ast})}}

\title{Marcinkiewicz-Zygmund inequalities in variable Lebesgue spaces}\thanks{This work was partially supported by CONICET PIP  11220200102366CO,
ANPCyT PICT 2018-04104,
and UNCOMA PIN I 04/B251. The first author has a doctoral fellowship from CONICET
}

\subjclass[2010] {}

\date{}

\begin{document}
\baselineskip=.65cm

\author[Bonich]{Marcos Bonich}
\address[Bonich]{IMAS (UBA-CONICET), Argentina.}
\email{\texttt{mjbonich@dm.uba.ar}}

\author[Carando]{Daniel Carando}
\address[Carando]{Departamento de Matem\'atica, Facultad de Cs. Exactas y Naturales, Universidad de
Buenos Aires and IMAS (UBA-CONICET), Argentina.}
\email{\texttt{dcarando@dm.uba.ar}}

\author[Mazzitelli]{Martin Mazzitelli}
\address[Mazzitelli]{Instituto Balseiro, CNEA-Universidad Nacional de Cuyo, CONICET, San Carlos de Bariloche, Argentina.}
\email{\texttt{martin.mazzitelli@ib.edu.ar}}

\begin{abstract}
We study $\ell^r$-valued extensions of linear operators defined on Lebesgue spaces with variable exponent. Under some natural (and usual) conditions on the exponents, we characterize $1\leq r\leq \infty$ such that every bounded linear operator $T\colon L^{q(\cdot)}(\Omega_2, \mu)\to L^{p(\cdot)}(\Omega_1, \nu)$ has a bounded $\ell^r$-valued extension. We consider both non-atomic measures and measures with atoms and show the differences that can arise. We present some applications of our results to weighted norm inequalities of linear operators and vector-valued extensions of fractional operators with rough kernel.
\end{abstract}



\maketitle

\section{Introduction}\label{intro}
Given a bounded linear operator $T\colon L^q(\Omega_2, \mu) \to L^p(\Omega_1, \nu)$, we  consider the \emph{natural} vector-valued extension given by
$$
\widetilde{T}\left((f_k)_k\right)= \left(T(f_k)\right)_k,
$$
where $(f_k)_k$ is any finite sequence of functions in $L^q(\Omega_2, \mu)$.
It is natural to ask if $\widetilde{T}$ extends to a  well defined bounded operator from $L^q(\mu,\ell^r)$ into $L^p(\nu,\ell^r)$.
By density, this is equivalent to asking if there exists a constant $C_T>0$ such that
\begin{equation}\label{eq-MaZy-1}
 \left\| \left(\sum_k |T(f_k)|^r\right)^{1/r} \right\|_{L^p(\Omega_1, \nu)}  \leq C_T  \left\| \left(\sum_k |f_k|^r\right)^{1/r} \right\|_{L^q(\Omega_2, \mu)}.
\end{equation}
for all $(f_k)_k \in L^q(\mu,\ell^r)$.
Vector-valued inequalities arise from the very beginning of harmonic analysis, closely related to the study of maximal and singular operators.
In what follows, we are interested in determining possible vector-valued extensions of bounded linear operators. We distinguish between two related problems:
\begin{itemize}
\item[(P1)] For which $(p,q,r)$ is it true that \emph{every} bounded linear operator $T\colon L^q(\Omega_2, \mu)\to L^p(\Omega_1, \nu)$ admits a bounded  $\ell^r$-valued extension?
\item[(P2)] Fixed an operator $T\colon L^q(\Omega_2, \mu)\to L^p(\Omega_1, \nu)$, for which $0< r\leq\infty$ does the operator $T$ have a bounded $\ell^r$-valued extension?
\end{itemize}
Marcinkiewicz and Zygmund were the first to study problem (P1). Note that, by a standard application of the closed graph theorem,  every $T\colon L^q(\Omega_2, \mu)\to L^p(\Omega_1, \nu)$ admits a bounded  $\ell^r$-valued extension if and only if \eqref{eq-MaZy-1} holds with $C_T = C\|T\|$ for some $C$ independent from $T$.
Marcinkiewicz and Zygmund proved in \cite{MarZyg} that given any $0<p,q<\infty$, every $T\colon L^q(\Omega_2, \mu)\to L^p(\Omega_1, \nu)$ admits an $\ell^2$-valued extension: there is a constant $C\geq 1$ such that
$$
 \left\| \left(\sum_k |T(f_k)|^2\right)^{1/2} \right\|_{L^p(\Omega_1, \nu)}  \leq C \|T\| \left\| \left(\sum_k |f_k|^2\right)^{1/2} \right\|_{L^q(\Omega_2, \mu)}.
$$
In \cite{MarZyg} it is also proved that if $0<\max\{p,q\}<r<2$, then there is a constant $C\geq 1$ such that
\begin{equation}\label{MZ}
 \left\| \left(\sum_k |T(f_k)|^r\right)^{1/r} \right\|_{L^p(\Omega_1, \nu)}  \leq C \|T\| \left\| \left(\sum_k |f_k|^r\right)^{1/r} \right\|_{L^q(\Omega_2, \mu)}.
\end{equation}
for every $T\colon L^q(\Omega_2, \mu)\to L^p(\Omega_1, \nu)$ (equivalently, $T$ admits an $\ell^r$-valued extension).

The infimum of all the constants $C\geq 1$ satisfying \eqref{MZ} for every $T\colon L^{q}(\Omega_2, \mu)\to L^{p}(\Omega_1, \nu)$ (for arbitrary $\sigma$-finite measure spaces $(\Omega_1, \nu)$ and $(\Omega_2, \mu)$), every $(f_k)_{k=1}^N\subset L^{q}(\Omega_2, \mu)$ and every $N\in \mathbb{N}$, is denoted by $k_{q,p}(r)$.
Many authors continued the line of study of Marcinkiewicz and Zygmund. In \cite{DefJun, GasMal}, a systematic study of the constants $k_{q,p}(r)$ is addressed for the full range $1\leq p,q,r \leq \infty$. We summarize the results of \cite{DefJun, GasMal} (see also \cite{RdFTor}) in the following theorem.
\begin{theorem}[\cite{DefJun, GasMal, RdFTor}]\label{teoDefJun}
Let $1\leq p,q,r \leq \infty$ and let $I(p,q)$ be the interval given by
$$I(p,q) = \left\{
\begin{array} {c l}
(q,2] & \text{if $p<q<2$,}\\
\left[2,p\right)  & \text{if $2<p<q$,}\\
\left[ \min\{2,q\}, \max\{2,p\} \right]  & \text{otherwise.}
\end{array}
\right .$$
\begin{enumerate}[label=(\roman*)]
\item If $q=1$ or $p=\infty$ then $k_{q,p}(r)=1$ for every $1\leq r \leq \infty$.
\item In any other case, $k_{q,p}(r)<\infty$ if and only if $r \in I(p,q)$.
\end{enumerate}
Moreover, the exact value of constant $k_{q,p}(r)$ is known, except for $1 \leq p < r=2 < q \leq \infty$.
\end{theorem}

The previous result characterizes all the exponents $1\leq p,q,r\leq \infty$ such that \emph{every} operator from $L^q(\Omega_2, \mu)$ into $L^p(\Omega_1, \nu)$ has a $\ell^r$-valued extension. Of course, when we consider some particular operator we can say more. For instance, it is not difficult to see that if $T\colon L^q(\Omega_2, \mu)\to L^p(\Omega_1, \nu)$ is a positive operator (that is, $T(f)\geq 0$ for every $f\geq 0$) then inequality \eqref{MZ} holds for all $1\leq r\leq \infty$, with $C=1$. The problem of obtaining vector-valued extensions of particular linear operators goes back to some works of Boas, Bochner and M. Riesz, among others, where  vector-valued extensions of the Hilbert transform are addressed. These early works motivated the study of vector-valued inequalities of Calder\'on-Zygmund operators in the line of problem (P2) stated above. Although there is a wide interest on these topics, including vector-valued inequalities for Calder\'on-Zygmund operators on variable Lebesgue spaces (see, for instance, \cite{CruFioMarPer, CruMarPer} and references therein), in this work we mainly focus on Problem (P1).
Our goal is to extend Theorem~\ref{teoDefJun} to the context of variable Lebesgue spaces. We begin by presenting the families of variable exponents considered in this work.

\begin{definition}\label{def p variable-pesos}
Let $(\Omega, \Sigma, \nu)$ be a $\sigma$-finite, complete measure space. We denote by $\mathcal{P}(\Omega,\nu)$ to the set of all $\nu$-measurable functions $p\colon \Omega \to [1,\infty]$. We also define:
$$
p_{_-}=\essinf_{x\in \Omega}p(x) \quad \text{and} \quad p_{_+}=\esssup_{x\in \Omega}p(x)
$$
and $\mathcal{P}_b(\Omega,\nu)=\{p\in \mathcal{P}(\Omega,\nu):\,\, 1<p_-\leq p_+<\infty\}$.
\end{definition}

Let us now recall the definition of Lebesgue spaces with variable exponent. First, let
$$
\Omega^{p(\cdot)}_\infty=\{x\in \Omega:\,\, p(x)=\infty\}
$$
and define the modular functional associated to $p(\cdot)$ by
$$
\rho_{p(\cdot)}(f)=\int_{\Omega\setminus \Omega^{p(\cdot)}_\infty}|f(x)|^{p(x)}d\nu(x) + \|f\|_{L^\infty(\Omega^{p(\cdot)}_\infty)}.
$$

\begin{definition}\label{def p variable-espacios}
We denote by $L^{p(\cdot)}(\Omega, \nu)$  the Banach space of $\nu$-measurable functions $f$ such that $\rho_{p(\cdot)}(f/\lambda)<\infty$ for some $\lambda>0$, endowed with the norm
$$
\|f\|_{L^{p(\cdot)}(\Omega, \nu)}=\inf\{\lambda>0:\,\, \rho_{p(\cdot)}(f/\lambda)\leq 1\}.
$$
If there is no ambiguity, we simply denote this norm by $\|\cdot\|_{L^{p(\cdot)}(\nu)}$ or $\|\cdot\|_{L^{p(\cdot)}(\Omega)}$.
\end{definition}

Variable Lebesgue spaces were introduced by Orlicz in the early `30's. However, it was not until the `50's that Nakano \cite{Nak1, Nak2} studied systematically some of their properties, as a special case of his theory of modular spaces. Independently, Tsenov and Sharapudinov \cite{Sha, Tse} discovered and studied variable Lebesgue spaces, and were the first to consider applications to harmonic analysis and the calculus of variations. In the last years, the interest on variable Lebesgue spaces has increased dramatically due to their different applications, such as image processing or modeling of electrorheological fluids (see, for instance, \cite{CheLevRao, DieHarHasRuz}). In the present article we contribute to  the development of the theory of linear operators on variable Lebesgue spaces.

\subsection{Preview of the results} We want to characterize the exponents  $p(\cdot)$, $q(\cdot)$ and $r$ such that every operator
$T\colon L^{q(\cdot)}(\Omega_2, \mu)\to L^{p(\cdot)}(\Omega_1, \nu)$ admits a bounded $\ell^r$-extension. As above, by density and a standard  application of the closed graph theorem, this is equivalent to the existence of $C>0$ such that
\begin{equation}\label{MZ p variable}
 \left\| \left(\sum_{k=1}^N |T(f_k)|^r\right)^{1/r} \right\|_{L^{p(\cdot)}(\Omega_1, \nu)}  \leq C \|T\| \left\| \left(\sum_{k=1}^N |f_k|^r\right)^{1/r} \right\|_{L^{q(\cdot)}(\Omega_2, \mu)}.
\end{equation}
holds for every such $T$ and every finite sequence $(f_k)_k$  in $ L^{q(\cdot)}(\Omega_2, \mu)$ . We denote by $k_{q(\cdot),p(\cdot)}(r)$ the infimum of all the constants $C\geq 1$ such that \eqref{MZ p variable} holds for every $T\colon L^{q(\cdot)}(\Omega_2, \mu)\to L^{p(\cdot)}(\Omega_1, \nu)$, every $(f_k)_{k=1}^N\subset L^{q(\cdot)}(\Omega_2, \mu)$ and every $N\in \mathbb{N}$ (we set $k_{q(\cdot),p(\cdot)}(r)=\infty$ if  no such constant exists). Whenever $k_{q(\cdot),p(\cdot)}(r)<\infty$, equation \eqref{MZ p variable} is equivalent to saying that the vector-valued extension of the operator $T$, given by
$$
\tilde{T}((f_k)_k)=(T(f_k))_k,
$$
is bounded from $L^{q(\cdot)}(\mu, \ell^r)$ to $L^{p(\cdot)}(\nu, \ell^r)$.  In Theorem~\ref{teoDefJunVariable} we prove for $p(\cdot)\in \mathcal{P}_b(\Omega_1,\nu)$, $q(\cdot)\in \mathcal{P}_b(\Omega_2,\mu)$ (where $\nu$ and $\mu$ are non-atomic measures) and $1\leq r\leq\infty$,
that $k_{q(\cdot), p(\cdot)}(r)<\infty$ if and only if $r \in I(p_-,q_+)$. Here, $I(p_-,q_+)$ is the interval defined in the statement of Theorem~\ref{teoDefJun}.
To do this, we follow ideas in \cite{DefJun, GasMal} to obtain a kind of monotonicity and duality properties for the constants $k_{q(\cdot),p(\cdot)}(r)$ that are the key to prove the ``if'' part of our main result. The ``only if'' part follows essentially by restricting the exponents $p(\cdot)$ and $q(\cdot)$ to subsets (of $\Omega_1$ and $\Omega_2$, respectively) where they  take values \emph{arbitrarily close} to $p_-$ and $q_+$, respectively.

As a direct application of Theorem~\ref{teoDefJunVariable} and the known relation between vector-valued inequalities and factorization theorems, we obtain in Corollary~\ref{application weighted ineq} some weighted inequalities for \emph{every} linear operator $T\colon L^{q(\cdot)}(\Omega_2, \mu)\to L^{p(\cdot)}(\Omega_1, \nu)$. For instance, we prove that if $2\leq r<\min\{p_-, q_-\}$, then given a weight $u\in L^{\frac{p(\cdot)}{p(\cdot)-r}}(\Omega_1, \nu)$ there exist a weight $U\in  L^{\frac{q(\cdot)}{q(\cdot)-r}}(\Omega_2, \mu)$ and a constant $K>0$ such that
\begin{equation*}
\int_{\Omega_1} |T(f)(x)|^r u(x)\, d\nu(x) \leq K \int_{\Omega_2} |f(x)|^r U(x)\, d\mu(x)
\end{equation*}
holds for every $f\in L^{q(\cdot)}(\Omega_2, \mu)$. We obtain an analogous result when $\max\{p_+, q_+\}< r\leq 2$.
In Corollary~\ref{vector valued fractional} we derive some vector-valued inequalities for fractional operators with rough kernel of variable order $\alpha(\cdot)$. These inequalities follow easily from Theorem~\ref{teoDefJunVariable} and the boundedness of the mentioned operators obtained in \cite{RafSam}.

Finally, in Theorem~\ref{teoDefJunAtomic} and Remark~\ref{ejemplos borde}
we see  that, for measure spaces $(\Omega_1, \nu)$ and $(\Omega_2,\mu)$ with atoms, the range of $r$'s for which the Marcinkiewicz-Zygmund inequalities hold can be larger that the one given in Theorem~\ref{teoDefJunVariable}.

\subsection{Structure of the article}
The article is organized as follows.
In Section~\ref{duality monotonicity} we prove the mentioned properties of the constants $k_{q(\cdot),p(\cdot)}(r)$.
In Proposition~\ref{duality formula variable} we prove a ``duality'' formula which is useful to prove decreasing and increasing properties of $k_{q(\cdot),p(\cdot)}(r)$ in the variables $p(\cdot), q(\cdot)$ and $r$. These results are stated in Theorem~\ref{decreasing and increasing}.
In Section~\ref{MaZy variable} we extend Theorem~\ref{teoDefJun} to the variable setting. The main result of this section (and of the article) is stated in Theorem~\ref{teoDefJunVariable}.
Section~\ref{applications} is devoted to the applications of Theorem~\ref{teoDefJunVariable} and in Section~\ref{Oliver Atom} we return to Marcinkiewicz-Zygmund type inequalities in the context of measure spaces with atoms.

All the Banach spaces considered are real (although the results can be extended to the complex case) and the measure spaces $(\Omega, \nu)$ are $\sigma$-finite. The notation is standard; we refer the reader to \cite{CruFio, DieHarHasRuz} as general references on variable Lebesgue spaces and \cite{CarMazOmb, DefJun, GasMal} as references on Marcinkiewicz-Zygmund inequalities.

\section{Monotonicity and duality properties of the constants $k_{q(\cdot),p(\cdot)}(r)$}\label{duality monotonicity}

The main goal of this section is to prove the following theorem on the increasing/decreasing dependence of $k_{q(\cdot),p(\cdot)}(r)$ on the parameters $q(\cdot),p(\cdot)$ and $r$.


\begin{theorem}\label{decreasing and increasing} \label{increasing in r}
	Let $p(\cdot), p_1(\cdot),p_2(\cdot)\in \mathcal{P}_b(\Omega_1, \nu)$ and  $q_1(\cdot), q_2(\cdot)\in \mathcal{P}_b(\Omega_2, \mu)$.  For arbitrary  $q(\cdot)\in \mathcal{P}(\Omega_2, \mu)$ and  $1\leq r < \infty$, the following assertions hold.
\begin{enumerate}[label=(\roman*)]	
\item If $1\leq r < s\leq 2$, then $k_{q(\cdot),p(\cdot)}(s)\leq k_{q(\cdot),p(\cdot)}(r)$.
\item If $2\leq r < s<\infty$ then  $k_{q(\cdot),p(\cdot)}(r)\leq 2^4\, k_{q(\cdot),p(\cdot)}(s)$.
\item If $p_2(\cdot)\leq p_1(\cdot)$ almost everywhere, then $k_{q(\cdot),p_1(\cdot)}(r)\le (2^3+1) \, k_{q(\cdot),p_2(\cdot)}(r)$.
\item If $q_1(\cdot)\leq q_2(\cdot)$ almost everywhere, then $k_{q_1(\cdot),p(\cdot)}(r)\le (2^7+2^4) \, k_{q_2(\cdot),p(\cdot)}(r)$.
\end{enumerate}	
\end{theorem}

As a first step, we show a duality property for the constants $k_{q(\cdot),p(\cdot)}(r)$ which we feel is interesting on its own and will be useful also in Section \ref{MaZy variable}.
Given $p\colon \Omega \to [1,\infty]$ we define the \emph{conjugate exponent} $p'\colon \Omega \to [1,\infty]$ as
$$
\frac{1}{p(x)}+\frac{1}{p'(x)}=1.
$$
For constant exponents, it is known that $k_{q,p}(r)=k_{p', q'}(r')$ for every $1\leq p, q, r \leq \infty$ (see \cite{DefJun, GasMal}). When $1<p,q<\infty$, this follows easily from the fact that $L^p(\Omega, \nu)$ and $L^q(\Omega, \nu)$ are isometrically isomorphic to $L^{p'}(\Omega, \nu)$ and $L^{q'}(\Omega, \nu)$, respectively.
We cannot expect such an  equality in the variable case since the isomorphism $(L^{p(\cdot)}(\Omega, \nu))'\cong L^{p'(\cdot)}(\Omega, \nu)$ is not necessarily isometric. Moreover, when $p_+=\infty$ (even if $p(x)<\infty$ for $\nu$-a.e $x\in \Omega$), there is no isomorphism between $(L^{p(\cdot)}(\Omega, \nu))'$ and $L^{p'(\cdot)}(\Omega, \nu)$ (see \cite[Corollary~2.7]{KovRak}). Let us define the norm
$$
\|f\|_{L^{p(\cdot)}(\nu)}'=\sup \int_\Omega |f(x)g(x)|d\nu(x),
$$
where the supremum is taken over all $g\in L^{p'(\cdot)}(\Omega, \nu)$ with $\|g\|_{L^{p'(\cdot)}(\nu)}\leq 1$. It is known that
\begin{equation}\label{normas equivalentes}
\frac{1}{2}\|f\|_{L^{p(\cdot)}(\nu)} \leq \|f\|_{L^{p(\cdot)}(\nu)}' \leq 2 \|f\|_{L^{p(\cdot)}(\nu)}
\end{equation}
for every $f\in L^{p(\cdot)}(\Omega, \nu)$ and every $p\in \mathcal{P}(\Omega, \nu)$. It is clear then that
$$
\Phi_g(f)=\int_{\Omega}f(x)g(x)d\nu(x)
$$
is a linear functional in $(L^{p(\cdot)}(\nu))'$ whenever $g\in L^{p'(\cdot)}(\nu)$. As we already mentioned, in  \cite{KovRak} the authors prove that the map $g\mapsto \Phi_g$ is an isomorphism if and only if $p_+<\infty$. Consequently, $L^{p(\cdot)}(\nu)$ is reflexive if and only if $1<p_-\leq p_+<\infty$.

\begin{proposition}\label{duality formula variable}
	Let $p(\cdot)\in \mathcal{P}_b(\Omega_1,\nu)$, $q(\cdot)\in \mathcal{P}_b(\Omega_2,\mu)$ and $1\leq r\leq\infty$. Then,
$$k_{p'(\cdot),q'(\cdot)}(r') \le 4 \,k_{q(\cdot),p(\cdot)}(r).$$
\end{proposition}

\begin{proof}
	Take a linear operator $T:L^{p'(\cdot)}(\Omega_1, \nu)\longrightarrow L^{q'(\cdot)}(\Omega_2, \mu)$ and consider its adjoint operator $T^*:(L^{q'(\cdot)}(\Omega_2, \mu))'\longrightarrow (L^{p'(\cdot)}(\Omega_1, \nu))'$. Since, by reflexivity, the spaces $(L^{p'(\cdot)}(\Omega_1, \nu))'$ and $(L^{q'(\cdot)}(\Omega_2, \mu))'$ are isomorphic to $L^{p(\cdot)}(\Omega_1, \nu)$ and $L^{q(\cdot)}(\Omega_2, \mu)$, respectively,  we think of $T^*$ as an operator from $L^{q(\cdot)}(\Omega_2, \mu)$ to $L^{p(\cdot)}(\Omega_1, \nu)$ with norm $\left\| T^* \right\|\leq 2 \left\| T \right\|$.
Now, the natural vector-valued extension of $T^*$, which we call $\widetilde{T^*}$, is bounded from $L^{q(\cdot)}(\mu, \ell^r)$ to $L^{p(\cdot)}(\nu, \ell^r)$ with
	\begin{equation*}
		 \|\widetilde{T^*} \|\leq k_{q(\cdot),p(\cdot)}(r) \left\| T^* \right\| \leq k_{q(\cdot),p(\cdot)}(r)\, 2 \left\| T \right\|.
	\end{equation*}
	Now we transpose again to obtain the bounded linear operator
	$(\widetilde{T^*})^*:(L^{p(\cdot)}(\mu, \ell^r))'\longrightarrow (L^{q(\cdot)}(\nu, \ell^r))'$. Since $p_+<\infty$ and $q_+<\infty$, by
 the duality formula for valued Bochner-Lebesgue spaces with variable exponent \cite[Theorem~2]{CheXu},
we can think the operator $(\widetilde{T^*})^*$ defined from $L^{p'(\cdot)}(\mu, \ell^{r'})$ to $L^{q'(\cdot)}(\nu, \ell^{r'})$, with its norm at most doubled. It is standard to check that the operator $(\widetilde{T^*})^*$ is the natural vector-valued extension of $T$.
  This proves that every operator $T:L^{p'(\cdot)}(\Omega_1, \nu)\longrightarrow L^{q'(\cdot)}(\Omega_2, \mu)$ has a bounded vector-valued extension $\widetilde{T}$, with $\|\widetilde{T}\|\leq 4\, k_{q(\cdot),p(\cdot)}(r) \|T\|$. As a consequence, $k_{p'(\cdot),q'(\cdot)}(r')\le 4\, k_{q(\cdot),p(\cdot)}(r)<\infty$.
\end{proof}

Now we state the following lemma. The proof follows line by line that of its constant exponent version, which can be seen in \cite[Equation (8)]{GasMal}
\begin{lemma}\label{lemma_of_inequality_with_integrals}
	Let $p(\cdot)\in \mathcal{P}(\Omega_1, \nu)$, $q(\cdot)\in \mathcal{P}(\Omega_2, \mu)$ and $1\leq r < \infty$. Then,
	\begin{equation*}
		\left\| \left(\int_{0}^{1}\Big|\sum_{k=1}^{N}T(f_k)g_k(t)\Big|^rdt\right)^{\frac{1}{r}} \right\|_{L^{p(\cdot)}(\Omega_1, \nu)}
		\leq
		k_{q(\cdot),p(\cdot)}(r)\left\| T \right\|\left\| \left(\int_{0}^{1}\Big|\sum_{k=1}^{N}f_kg_k(t)\Big|^rdt\right)^{\frac{1}{r}} \right\|_{L^{q(\cdot)}(\Omega_2, \mu)}
	\end{equation*}
	for every bounded linear operator $T:L^{q(\cdot)}(\Omega_2, \mu)\to L^{p(\cdot)}(\Omega_1, \nu)$, every $\{f_k\}_{k=1}^N\subset L^{q(\cdot)}(\Omega_2, \mu)$ and $\{g_k\}_{k=1}^N\subset L^{r}[0,1]$ and every $N\in \mathbb{N}$.
\end{lemma}

We also state the following well known  result, which can be found in \cite[21.1.3]{Pie}.
	\begin{lemma}\label{lemma_of_integration}
		For $0<s<r<2$ or $r=2$ and $0<s<\infty$, let  $\{ w_k \}_k$ be  a sequence of $r$-stable independent random variables defined on $\left[ 0,1 \right]$ and let $c_{r,s}$ be its $s$-th moment. Then,
		\begin{equation*}
			\left( \int_{0}^{1} \left| \sum_{k=1}^{N}a_kw_k(t) \right|^sdt  \right)^{\frac{1}{s}}=c_{r,s}\left( \sum_{k=1}^{N} \left| a_k \right|^r \right)^{\frac{1}{r}}
		\end{equation*}
		for every sequence $\{a_k\}_{k=1}^N\subset \mathbb{R}$ and any $N\in \mathbb{N}$.
	\end{lemma}

Now we are ready to prove the main result of this section. For the dependence on $r$ we follow the ideas in the proof of \cite[Theorem 1 (c)]{GasMal}, where this is proved for constant exponents.
\begin{proof}[Proof of Theorem \ref{increasing in r}]
Le us prove item $(i)$. By Lemmas~\ref{lemma_of_integration} and \ref{lemma_of_inequality_with_integrals} we have
	\begin{eqnarray*}
		c_{s,r}\left\|\left( \sum_{k=1}^{N} \left| T(f_k) \right|^s \right)^{\frac{1}{s}}\right\|_{L^{p(\cdot)}(\Omega_1, \nu)}
		&=&
		\left\|\left( \int_{0}^{1} \left| \sum_{k=1}^{N}T(f_k)w_k(t) \right|^rdt  \right)^{\frac{1}{r}}\right\|_{L^{p(\cdot)}(\Omega_1, \nu)}
		\\
		&\leq&
		k_{q(\cdot),p(\cdot)}(r)\left\| T \right\|\left\| \left(\int_{0}^{1}\left|\sum_{k=1}^{N}f_kw_k(t)\right|^rdt\right)^{\frac{1}{r}} \right\|_{L^{q(\cdot)}(\Omega_2, \mu)}
		\\
		&=&
		k_{q(\cdot),p(\cdot)}(r)\left\| T \right\|c_{s,r}\left\| \left(\sum_{k=1}^{N}\left|f_k\right|^s\right)^{\frac{1}{s}} \right\|_{L^{q(\cdot)}(\Omega_2, \mu)}.
		\end{eqnarray*}
	Then
	\begin{equation*}
		\left\|\left( \sum_{k=1}^{n} \left| T(f_k) \right|^s \right)^{\frac{1}{s}}\right\|_{L^{p(\cdot)}(\Omega_1, \nu)}
		\leq
		k_{q(\cdot),p(\cdot)}(r)\left\| T \right\|\left\| \left(\sum_{k=1}^{N}\left|f_k\right|^s\right)^{\frac{1}{s}} \right\|_{L^{q(\cdot)}(\Omega_2, \mu)}
	\end{equation*}
	and, consequently, $k_{q(\cdot),p(\cdot)}(s)\leq k_{q(\cdot),p(\cdot)}(r)$. Item $(ii)$ follows from item $(i)$ and Proposition~\ref{duality formula variable}.

In order to prove $(iii)$, we first suppose  that $p_2(\cdot)<p_1(\cdot)$ almost everywhere, that is, the inequality is strict. Let  $T\colon L^{q(\cdot)}(\Omega_2, \mu) \to L^{p_1(\cdot)}(\Omega_1, \nu)$ be a linear operator and $(f_k)_k\subset L^{q(\cdot)}(\Omega_2, \mu)$.
By \eqref{normas equivalentes} we have
	\begin{equation}\label{dualidad}
		\left\| \left(\sum_k |T(f_k)|^r \right)^{\frac{1}{r}}\right\|_{L^{p_1(\cdot)}(\Omega_1,\nu)}
		\leq 2
		\sup_{\left\|g \right\|_{{p_1'(\cdot)}} \leq 1}\Biggl\{ \underbrace{\int_{\Omega_1} \left(\sum_k |T(f_k)|^r \right)^{\frac{1}{r}}g(x)\,d\nu(x)}_{\clubsuit} \Biggr\}.
	\end{equation}
Consider $\alpha(\cdot)=\frac{p_1'(\cdot)}{p_2(\cdot)}(1-\frac{p_2(\cdot)}{p_1(\cdot)})$ and $\beta(\cdot)=1-\alpha(\cdot)$. It is not difficult to see that $0<\alpha(\cdot),\beta(\cdot)<1$.
We write
	\begin{equation*}
		g(x)=|g(x)|\text{sg}(g(x))=|g(x)|^{\alpha(x)}|g(x)|^{\beta(x)}\text{sg}(g(x)),
	\end{equation*}
where sg denotes de sign function. Therefore,
	\begin{equation*}
		\clubsuit
		=
		\int_{\Omega_1} \left(\sum_k \Big|T(f_k)(x)|g(x)|^{\alpha(x)}\Big|^r \right)^{\frac{1}{r}}|g(x)|^{\beta(x)}\text{sg}(g(x))\,d\nu(x).
	\end{equation*}
	Now, on the one hand, let us see that
	\begin{equation*}
		T_g(f)(x)=|g(x)|^{\alpha(x)}T(f)(x)
	\end{equation*}
	defines a bounded linear operator from $L^{q(\cdot)}(\Omega_2, \mu)$ into $L^{p_2(\cdot)}(\Omega_1, \nu)$. Indeed, by H\"older's inequality for variable Lebesgue spaces (see \cite[Theorem~2.1]{KovRak}) we have
	\begin{equation*}
		\left\| T_g (f) \right\|_{L^{p_2(\cdot)}(\Omega_1, \nu)}
		\leq
		2\left\| T (f) \right\|_{L^{p_1(\cdot)}(\Omega_1, \nu)}\left\| |g|^{\alpha} \right\|_{L^{\theta(\cdot)}(\Omega_1, \nu)}
		\leq
		2 \|T\| \left\| |g|^{\alpha} \right\|_{L^{\theta(\cdot)}(\Omega_1, \nu)} \|f\|_{L^{q(\cdot)}(\Omega_2, \mu)}
	\end{equation*}
	where $\theta(\cdot)=\frac{p_2(\cdot)p_1(\cdot)}{p_1(\cdot)-p_2(\cdot)}$.
Since $\left\| g \right\|_{L^{p'_1(\cdot)}(\Omega_1, \nu)}\leq 1$ and  $\alpha(\cdot) \theta(\cdot)= p'_1(\cdot)$, we get $\left\| |g|^{\alpha} \right\|_{L^{\theta(\cdot)}(\Omega_1, \nu)}\leq 1$. Consequently, $T_g$ is bounded with $\|T_g\|\leq 2\|T\|$.

On the other hand, if $$h_g(x)=|g(x)|^{\beta(x)} \text{sg}(g(x)),$$ we can see that $h_g\in L^{p'_2(\cdot)}(\Omega_1, \nu)$ with $\left\| h_g \right\|_{L^{p'_2(\cdot)}(\Omega_1, \nu)}\leq 1$. Indeed, since $\rho_{p'_1(\cdot)}(g)\leq 1$ and $\beta(\cdot)p_2'(\cdot)=p_1'(\cdot)$, we have
	\begin{eqnarray*}
		\rho_{p'_2(\cdot)}(h_g)
		=
		\int_{\Omega_1} |g(x)|^{\beta(x)p'_2(x)}\, d\nu(x)
		=
		\int_{\Omega_1} |g(x)|^{p'_1(x)}\, d\nu(x)
		\leq
		\rho_{p'_1(\cdot)}(g)\leq 1
	\end{eqnarray*}
and, consequently,
	\begin{equation*}
		\left\| h_g \right\|_{p'_2(\cdot)}=\inf \{ \lambda>0 : \rho_{p'_2(\cdot)}(h/\lambda) \leq 1 \} \leq 1.
	\end{equation*}
Going back to \eqref{dualidad} we obtain
	\begin{eqnarray}
	\nonumber \left\| \left(\sum_k |T(f_k)|^r \right)^{\frac{1}{r}}\right\|_{L^{p_1(\cdot)}(\Omega_1, \nu)}&\leq&
		2\sup_{\left\|g \right\|_{p_1'(\cdot)} \leq 1}\Biggl\{ \int_{\Omega_1} \left(\sum_k |T(f_k)|^r \right)^{\frac{1}{r}}g(x)\,d\nu(x)\Biggr\}\\
		&=&
		2\sup_{\left\|g \right\|_{p_1'(\cdot)} \leq 1}\Biggl\{ \int_{\Omega_1} \left(\sum_k |T_g(f_k)|^r \right)^{\frac{1}{r}}h_g(x)\,d\nu(x) \Biggr\} \label{sup en g}
\end{eqnarray}
Now, for each $\left\|g \right\|_{p_1'(\cdot)}\leq 1$ we have
\begin{eqnarray*}
	\int_{\Omega_1} \left(\sum_k |T_g(f_k)|^r \right)^{\frac{1}{r}}h_g(x)\,d\nu(x)	&\leq&
		\sup_{\left\|h \right\|_{p_2'(\cdot)} \leq 1}\Biggl\{ \int_{\Omega_1} \left(\sum_k |T_g(f_k)|^r \right)^{\frac{1}{r}}h(x)\,d\nu(x) \Biggr\}\\
		&\leq& 2 \left\| \left(\sum_k |T_g(f_k)|^r \right)^{\frac{1}{r}} \right\|_{L^{p_2(\cdot)}(\Omega_1, \nu)} \quad \text{(by \eqref{normas equivalentes})}\\
		&\leq&
		2k_{q(\cdot),p_2(\cdot)}(r) \, 2 \left\| T \right\| \left\| \left(\sum_{k}|f_k |^r\right)^{\frac{1}{r}} \right\|_{L^{q(\cdot)}(\Omega_2, \mu)}.
	\end{eqnarray*}
Then, taking supremum over $\left\|g \right\|_{p_1'(\cdot)}\leq 1$ and returning to \eqref{sup en g} we get
$$
 \left\| \left(\sum_k |T(f_k)|^r \right)^{\frac{1}{r}}\right\|_{L^{p_1(\cdot)}(\Omega_1, \nu)}\leq 2^3\, k_{q(\cdot),p_2(\cdot)}(r)  \left\| T \right\| \left\| \left(\sum_{k}|f_k |^r\right)^{\frac{1}{r}} \right\|_{L^{q(\cdot)}(\Omega_2, \mu)},
$$
which proves that $k_{q(\cdot),p_1(\cdot)}(r)\leq 2^3\, k_{q(\cdot),p_2(\cdot)}(r)$.

To finish $(iii)$, we assume  $p_2(\cdot)\leq p_1(\cdot)$ almost everywhere,  consider
$
\widetilde{\Omega_1}=\{x\in \Omega_1:\,\, p_2(x)<p_1(x)\}
$
and split the operator $T\colon L^{q(\cdot)}(\Omega_2, \mu)\to L^{p_1(\cdot)}(\Omega_1, \nu)$ as follows
$$T(f)(x)=T(f)(x)\chi_{\widetilde{\Omega_1}}(x) + T(f)(x)\chi_{\widetilde{\Omega_1}^c}(x)=:T_{\widetilde{\Omega_1}}(f)(x)+ T_{\widetilde{\Omega_1}^c}(f)(x).$$
Then, we use what  we proved for the strict inequality together with Lemma~\ref{lema 1} below to get
\begin{eqnarray*}
\left\| \left(\sum_k |T(f_k)|^r \right)^{\frac{1}{r}}\right\|_{L^{p_1(\cdot)}(\Omega_1)}
&\leq&
\left\| \left(\sum_k |T_{\widetilde{\Omega_1}}(f_k)|^r \right)^{\frac{1}{r}}\right\|_{L^{p_1(\cdot)}(\widetilde{\Omega_1})}+ \left\| \left(\sum_k |T_{\widetilde{\Omega_1}^c}(f_k)|^r \right)^{\frac{1}{r}}\right\|_{L^{p_1(\cdot)}(\widetilde{\Omega_1}^c)}\\
&\leq& (2^3+1)  k_{q(\cdot), p_2(\cdot)}(r)  \|T\| \left\| \left(\sum_{k}|f_k |^r\right)^{\frac{1}{r}} \right\|_{L^{q(\cdot)}(\Omega_2)},
\end{eqnarray*}
which proves the desired statement.

To show $(iv)$, fix $p(\cdot)\in  \mathcal{P}_b(\Omega_1, \nu)$ and suppose that $q_1(\cdot)\leq q_2(\cdot)$. Since we clearly have that $q_2'(\cdot)\leq q_1'(\cdot)$, by Proposition~\ref{duality formula variable} and $(iii)$ we have
	\begin{eqnarray*}
		k_{q_1(\cdot),p(\cdot)}(r) \le 2^2 \,k_{p'(\cdot),q'_1(\cdot)}(r') \le (2^5+2^2) \, k_{p'(\cdot),q'_2(\cdot)}(r')
		\le (2^7+2^4) \, k_{q_2(\cdot),p(\cdot)}(r).
	\end{eqnarray*}
	This completes the proof.
\end{proof}

\section{Marcinkiewicz-Zygmund inequalities in variable Lebesgue spaces for non-atomic measure spaces}\label{MaZy variable}
Now we are ready to state our main theorem, which extends Theorem \ref{teoDefJun} to variable exponents defined on non-atomic measure spaces.
Recall from Theorem \ref{teoDefJun} the definition, for constant exponents $p,q$, of the intervals   $I(p,q)$:
$$I(p,q) = \left\{
\begin{array} {c l}
(q,2] & \text{if $p<q<2$,}\\
\left[2,p\right)  & \text{if $2<p<q$,}\\
\left[ \min\{2,q\}, \max\{2,p\} \right]  & \text{otherwise.}
\end{array}
\right .$$
\begin{theorem}\label{teoDefJunVariable}
	Let $(\Omega_1,\nu)$ and $(\Omega_2,\mu)$ be non-atomic measure spaces. Let $p(\cdot)\in \mathcal{P}_b(\Omega_1,\nu)$, $q(\cdot)\in \mathcal{P}_b(\Omega_2,\mu)$ and $1\leq r\leq\infty$.
	Then $k_{q(\cdot), p(\cdot)}(r)<\infty$ if and only if $r \in I(p_-,q_+)$.
Also, $k_{1, p(\cdot)}(r)<\infty$ for every $1\leq r<\infty$ and $k_{\infty, p(\cdot)}(r)<\infty$ for every $r\in I(p_-,\infty)$.
	\end{theorem}

The ``if'' part of the first statement  follows from the monotonicity properties proved in Section~\ref{duality monotonicity} for the constants $k_{q(\cdot), p(\cdot)}(r)$. %
For the reverse implication we need some auxiliary results. In what follows, given measure spaces $(\Omega_1,\nu)$ and $(\Omega_2,\mu)$, we denote by $k_{L^{q(\cdot)}(\Omega_2), L^{p(\cdot)}(\Omega_1)}(r)$  the constant $k_{q(\cdot), p(\cdot)}(r)$, in order to make explicit the domain in which the variable exponents $p(\cdot), q(\cdot)$ are defined. The first lemma shows that the constants decrease when we restrict ourselves to (measurable) subsets of $\Omega_1$ and $\Omega_2$.

\begin{lemma}\label{lema 1}
Let $(\Omega_1,\nu)$, $(\Omega_2,\mu)$ be measure spaces and $p(\cdot)\in \mathcal{P}_b(\Omega_1,\nu)$, $q(\cdot)\in \mathcal{P}_b(\Omega_2,\mu)$. For all measurable subsets  $\widetilde{\Omega_1} \subset\Omega_1$ and $\widetilde{\Omega_2}\subset \Omega_2$ we have $$k_{L^{q(\cdot)}(\widetilde{\Omega_2}), L^{p(\cdot)}(\widetilde{\Omega_1})}(r)\leq k_{L^{q(\cdot)}(\Omega_2), L^{p(\cdot)}(\Omega_1)}(r).$$
\end{lemma}

\begin{proof}
Take a linear operator $S\colon L^{q(\cdot)}(\widetilde{\Omega_2}, \mu)\to L^{p(\cdot)}(\widetilde{\Omega_1}, \nu)$ and define $T_S\colon L^{q(\cdot)}(\Omega_2,\mu)\to L^{p(\cdot)}(\Omega_1,\nu)$ by
	\begin{equation*}
		T_S(f)=S(f\cdot\chi_{\widetilde{\Omega_2}}),
	\end{equation*}
which is clearly a bounded linear operator with $\|T_S\|\leq \|S\|$.	

Now, let $\left(f_k\right)_k\subset L^{q(\cdot)}(\widetilde{\Omega_2}, \mu)$ and, for each $k$,  define  $g_k\in L^{q(\cdot)}(\Omega_2, \mu)$ by extending $f_k$ as zero outside $\widetilde{\Omega_2}$. Note that
\begin{eqnarray*}
\left\| \left(\sum_k \left|S(f_k)\right|^r \right)^{\frac{1}{r}}\right\|_{L^{p(\cdot)}(\widetilde{\Omega_1}, \nu)}
&=&
\left\| \left(\sum_k \left|S(g_k\cdot \chi_{\widetilde{\Omega_2}})\right|^r \right)^{\frac{1}{r}}\right\|_{L^{p(\cdot)}(\Omega_1, \nu)}
\\
&=&
\left\| \left(\sum_k \left|T_S(g_k)\right|^r \right)^{\frac{1}{r}}\right\|_{L^{p(\cdot)}(\Omega_1, \nu)}
\\
&\leq& k_{L^{q(\cdot)}(\Omega_2),L^{p(\cdot)}(\Omega_1)}(r) \left\|T_S\right\|	\left\| \left(\sum_k \left|g_k\right|^r \right)^{\frac{1}{r}}\right\|_{L^{q(\cdot)}(\Omega_2, \nu)}
\\
&\leq& k_{L^{q(\cdot)}(\Omega_2),L^{p(\cdot)}(\Omega_1)}(r) \left\|S\right\|	\left\| \left(\sum_k \left|f_k\right|^r \right)^{\frac{1}{r}}\right\|_{L^{q(\cdot)}(\widetilde{\Omega_2}, \nu)}.
\end{eqnarray*}
Consequently, $k_{L^{q(\cdot)}(\widetilde{\Omega_2}),L^{p(\cdot)}(\widetilde{\Omega_1})}(r)\leq k_{L^{q(\cdot)}(\Omega_2),L^{p(\cdot)}(\Omega_1)}(r)$.
\end{proof}

Recall form the Section \ref{intro} that, for $1\le p,q\le\infty $ (constant exponents),  $k_{q,p}(r)$ is defined as the infimum of all the constants $C>1$ such that
\begin{equation*}
 \left\| \left(\sum_{k=1}^N |T(f_k)|^r\right)^{1/r} \right\|_{L^p(\Omega_1, \nu)}  \leq C \|T\| \left\| \left(\sum_{k=1}^N |f_k|^r\right)^{1/r} \right\|_{L^q(\Omega_2, \mu)}
\end{equation*}
holds for every operator $T\colon L^q(\Omega_2, \mu)\to L^p(\Omega_1, \nu)$, every $(f_k)_{k=1}^N \in L^q(\Omega_2, \mu)$, every $N\in \mathbb{N}$ and \emph{every} measure spaces $(\Omega_1, \nu), (\Omega_2, \mu)$. In \cite{DefJun, Jun}  it is observed that it suffices to look at $C$'s satisfying the inequality for two fixed measure spaces $(\Omega_1^0, \nu_0)$ and $(\Omega_2^0, \mu_0)$ (whenever $L^p(\Omega_1^0, \nu_0)$ and $L^q(\Omega_2^0, \mu_0)$ are infinite-dimensional). We state this result in the following lemma.
\begin{lemma}[\cite{DefJun, Jun}]\label{lema 2}
Fix two measure spaces $(\Omega_1^0, \nu_0)$ and $(\Omega_2^0, \mu_0)$ which are not  union of finitely many atoms. Then,
$
k_{q,p}(r)=k_{L^q(\Omega_2^0), L^p(\Omega_1^0)}(r).
$
\end{lemma}

Now we are ready to prove Theorem~\ref{teoDefJunVariable}.
 In the sequel, we write $\lesssim$ to mean that inequality holds up to a universal constant. For example, $k_{q(\cdot), p(\cdot)}(r)\lesssim k_{q_+, p_-}(r)$ means that  there exist a universal constant $C>0$, independent of the exponents, the measure spaces and $r$, such that $k_{q(\cdot), p(\cdot)}(r)\le C k_{q_+, p_-}(r)$.

\begin{proof}[Proof of Theorem~\ref{teoDefJunVariable}]
Suppose first that $r \in I(p_- ,q_+)$. By Theorem \ref{teoDefJun} and the monotonicity properties proved in Theorem \ref{decreasing and increasing}, we have
$$
k_{q(\cdot), p(\cdot)}(r)\lesssim k_{q_+, p_-}(r)<\infty.
$$
Now suppose that $k_{q(\cdot), p(\cdot)}(r)=k_{L^{q(\cdot)}(\Omega_2), L^{p(\cdot)}(\Omega_1)}(r)<\infty$ and let us show that $r \in I(p_-, q_+)$. For any $\epsilon>0$ define the subsets
\begin{equation*}
		\Omega_2^{q_+-\epsilon}=\lbrace x\in\Omega_2:\,\,q(x)>q_+-\epsilon\rbrace
	\quad
	\text{and}
	\quad
		\Omega_1^{p_-+\epsilon}=\lbrace x\in\Omega_1:\,\,p(x)<p_-+\epsilon\rbrace.
	\end{equation*}
Notice that  $\mu(\Omega_2^{q_+-\epsilon})>0$ and $\nu(\Omega_1^{p_-+\epsilon})>0$ and that $(\Omega_1^{p_-+\epsilon}, \nu)$ and $(\Omega_2^{q_+-\epsilon}, \mu)$ are not union of finitely many atoms (since $(\Omega_1, \nu)$ and $(\Omega_2, \mu)$ are non-atomic).
By Lemma~\ref{lema 2}, monotonicity and Lemma~\ref{lema 1} we have
\begin{eqnarray}\label{desigualdad constante-variable}
\nonumber k_{q_+-\epsilon, p_-+\epsilon}(r)&=&k_{L^{q_+-\epsilon}(\Omega_2^{q_+-\epsilon}), L^{p_-+\epsilon}(\Omega_1^{p_-+\epsilon})}(r)\\
\nonumber &\lesssim&  k_{L^{q(\cdot)}(\Omega_2^{q_+-\epsilon}), L^{p(\cdot)}(\Omega_1^{p_-+\epsilon})}(r)\\
&\lesssim& k_{L^{q(\cdot)}(\Omega_2), L^{p(\cdot)}(\Omega_1)}(r)<\infty.
\end{eqnarray}
This holds for any $\epsilon>0$  and then we have
\begin{equation}\label{eq-clausuraI}
r\in\bigcap_{\epsilon>0}I(p_-+\epsilon,q_+-\epsilon).
 \end{equation}
The rest of the proof is split depending on the values of  $p_-$ and $ q_+$.

\noindent \emph{Case $p_-<q_+<2$.} In what follows, note that if $\epsilon>0$ is sufficiently small, then $p_-+\epsilon < q_+ -\epsilon$. By \eqref{eq-clausuraI} and Theorem \ref{teoDefJun},  we know that  $r$ belongs to $\bigcap_{\epsilon>0}I(p_-+\epsilon,q_+-\epsilon) =\bigcap_{\epsilon>0}(q_+-\epsilon,2]= [q_+, 2]$, so we only have to show that $r$ cannot be equal to $q_+$.
In \cite[Section~4]{DefJun} it is proved that, for $\epsilon>0$ sufficiently small, we have
$$
k_{q_+-\epsilon, p_-+\epsilon}(r)=\frac{c_{r,q_+-\epsilon}}{c_{r,p_-+\epsilon}},
$$
where $c_{r,p}$ is the \emph{pth moment} of the $r$-stable L\'evy measure on $\mathbb{R}$ given by
$$
c_{r,p}=\left(\frac{\Gamma\left(\frac{r-p}{r}\right) \Gamma\left(\frac{1+p}{2}\right)}{\Gamma\left(\frac{2-p}{2}\right) \Gamma\left(\frac{1}{2}\right)}\right)^{1/p}.
$$
If $r=q_+$, as in \eqref{desigualdad constante-variable} we have for $\epsilon>0$ sufficiently small
$$
k_{L^{q(\cdot)}(\Omega_2), L^{p(\cdot)}(\Omega_1)}(q_+) \gtrsim k_{q_+-\epsilon, p_-+\epsilon}(q_+)=\frac{c_{q_+,q_+-\epsilon}}{c_{q_+,p_-+\epsilon}} \underset{\epsilon\to 0}{\longrightarrow}+\infty.
$$
Since $ k_{L^{q(\cdot)}(\Omega_2), L^{p(\cdot)}(\Omega_1)}(r)$ is finite,  we conclude that $r$ cannot be equal to $q_+$.

\noindent \emph{Case $2<p_-<q_+$.} This case follows from the previous one and a duality argument. On the one hand, note that if $2<p_-<q_+$ then $(q_+)'<(p_-)'<2$ and, since $(q_+)'=(q')_-$ and $(p_-)'=(p')_+$, we have $(q')_-<(p')_+<2$. On the other hand, since $k_{q(\cdot), p(\cdot)}(r)<\infty$, by Proposition~\ref{duality formula variable} we have $k_{p'(\cdot), q'(\cdot)}(r')<\infty$. Hence, by the previous case we know that $r'\in I(q'(\cdot), p'(\cdot))=((p')_+, 2]$ and, consequently, $r\in [2, ((p')_+)')=[2, p_-)$, which is the desired statement.

\noindent \emph{Case $p_-\leq 2<q_+$.} 
Take  $\epsilon>0$ such that $ 2+\epsilon <q_+-\epsilon$. Since $$k_{q_+-\epsilon, 2+\epsilon}(r)\le k_{q_+-\epsilon, p_-+\epsilon}(r)\lesssim k_{q(\cdot), p(\cdot)}(r) < \infty,$$
by  Theorem~\ref{teoDefJun} we have $r\in [2,2+\epsilon)$. Since this holds for every  $\epsilon> 0$ (sufficiently small) we see that $r=2$.

\noindent \emph{Case $p_-< 2\leq q_+$.} This case is analogous to the previous one.

\noindent \emph{Case $q_+\leq p_-$.}  Since $q_+-\epsilon<p_-+\epsilon$ for any $\epsilon>0$, by \eqref{eq-clausuraI} and  Theorem~\ref{teoDefJun}, we have $$r\in \bigcap_{\epsilon>0}[\min\{2, q_+-\epsilon\}, \max\{2, p_-+\epsilon\}]=[\min\{2, q_+\}, \max\{2, p_-\}]=  I(p_-, q_+). $$
Finally, the statements for $q=\infty$ and $q=1$  follow easily from the monotonicity in $p(\cdot)$ shown in part $(iii)$ of Theorem \ref{decreasing and increasing} and Theorem~\ref{teoDefJun}.
\end{proof}

\section{Applications}\label{applications}

In this section we apply Theorem~\ref{teoDefJunVariable} to characterize boundedness of linear operators defined on variable Lebesgue spaces in terms of some weighted inequalities and to obtain vector-valued inequalities for fractional operators with rough kernel. For simplicity, all the measure spaces considered here are non-atomic, although weighted inequalities can be obtained also for measures with atoms, applying Theorem~\ref{teoDefJunAtomic} below instead of Theorem~\ref{teoDefJunVariable}.

\subsection{Weighted norm inequalities}

It is a well-known fact that vector-valued inequalities are closely related to weighted norm inequalities and factorization theorems.
In \cite[Chapter~VI.5]{GarRub} it is shown that, under certain hypotheses on the parameters $1\leq p, q, r<\infty$, $\ell^r$-valued extensions of operators $T\colon L^q(\Omega_2, \mu)\to L^p(\Omega_1, \nu)$ are equivalent to some weighted $L^r$-inequalities. In \cite[Section~3]{Gar} (see also \cite{SanTra}) the same type of results were extended to the context of operators between $r$-convex and $r$-concave Banach lattices with absolutely continuous norm. Recall that a Banach lattice $X$ is \emph{$r$-convex} if there is a constant $M>0$ such that
$$
\left\| \left( \sum_k |x_k|^r\right)^{1/r}\right\|_X\leq M \left( \sum_k \|x_k\|_X^r\right)^{1/r}
$$
for every finite sequence $(x_k)_k\subset X$, and is  \emph{$r$-concave} if there is a constant $M>0$ such that
$$
\left( \sum_k \|x_k\|_X^r\right)^{1/r} \leq M \left\| \left( \sum_k |x_k|^r\right)^{1/r}\right\|_X
$$
for every finite sequence $(x_k)_k\subset X$. The space $L^{p(\cdot)}(\Omega, \nu)$ is $r$-convex if $r\leq p(\cdot)$ $\nu$-a.e. and is $r$-concave if $p(\cdot) \leq r$ $\nu$-a.e.  (see \cite[Propositions~3~and~4]{Kam}, where the result is stated in the more general setting of Musielak-Orlicz spaces). Also,  $L^{p(\cdot)}(\Omega, \nu)$ has absolutely continuous norm if  $p(\cdot)\in \mathcal{P}_b(\Omega,\nu)$ (see \cite{EdmLanNek}). Therefore,  we can apply \cite[Theorem~3.1~and~3.3]{Gar} to obtain the equivalence between $\ell^r$-valued inequalities and weighted $L^r$-inequalities in the particular case of operators between variable Lebesgue spaces. We state this equivalence in the next proposition. Let us recall that an operator $T\colon L^{q(\cdot)}(\Omega_2, \mu)\to L^{p(\cdot)}(\Omega_1, \nu)$ has an $\ell^r$-valued (bounded) extension if \eqref{MZ p variable} holds.
In what follows, we denote $L_+^{p(\cdot)}(\Omega, \nu)=\{f\in L^{p(\cdot)}(\Omega, \nu):\,\, f> 0\}$.

\begin{proposition}[\cite{Gar}]\label{weighted norm thm}
Let $(\Omega_1,\nu)$ and $(\Omega_2,\mu)$ be measure spaces, $p(\cdot)\in \mathcal{P}_b(\Omega_1,\nu)$, $q(\cdot)\in \mathcal{P}_b(\Omega_2,\mu)$, $1< r<\infty$ and define
$$
\frac{1}{\alpha(\cdot)}=\left| 1-\frac{r}{p(\cdot)}\right| \quad \text{and}\quad \frac{1}{\beta(\cdot)}=\left| 1-\frac{r}{q(\cdot)}\right|.
$$
Let $T\colon L^{q(\cdot)}(\Omega_2, \mu)\to L^{p(\cdot)}(\Omega_1, \nu)$ be a linear operator.
\begin{enumerate}[label=(\roman*)]
\item If $r< p(\cdot)$ $\nu$-a.e. and $r<q(\cdot)$ $\mu$-a.e., then \eqref{MZ p variable} holds if and only if, for each $u\in L_+^{\alpha(\cdot)}(\Omega_1, \nu)$, there exist $U\in L_+^{\beta(\cdot)}(\Omega_2, \mu)$ such that $\|U\|_ {\beta(\cdot)}\leq \|u\|_ {\alpha(\cdot)}$ and
\begin{equation*}
\int_{\Omega_1} |T(f)(x)|^r u(x)\, d\nu(x) \leq K \int_{\Omega_2} |f(x)|^r U(x)\, d\mu(x)
\end{equation*}
for some constant $K>0$ for every $f\in L^{q(\cdot)}(\Omega_2, \mu)$.

\item If $r>p(\cdot)$ $\nu$-a.e. and $r>q(\cdot)$ $\mu$-a.e., then \eqref{MZ p variable} holds if and only if, for each $u\in L_+^{\beta(\cdot)}(\Omega_2, \mu)$, there exist $U\in L_+^{\alpha(\cdot)}(\Omega_1, \nu)$ such that $\|U\|_ {\alpha(\cdot)}\leq \|u\|_ {\beta(\cdot)}$ and
\begin{equation*}
\int_{\Omega_1} |T(f)(x)|^r U^{-1}(x)\, d\nu(x) \leq K \int_{\Omega_2} |f(x)|^r u^{-1}(x)\, d\mu(x)
\end{equation*}
for some constant $K>0$ for every $f\in L^{q(\cdot)}(\Omega_2, \mu)$.
\end{enumerate}
\end{proposition}
In the case $T\colon L^{p(\cdot)}(\Omega_1, \nu)\to L^{p(\cdot)}(\Omega_1, \nu)$, an iteration argument can be applied in order to unify the weights $u$ and $U$ appearing in the previous result. This is stated and proved in \cite[Theorem~3.6~and~3.8]{Gar}.

\begin{proposition}[\cite{Gar}]\label{coro peso unificado}
We follow the notation in Proposition~\ref{weighted norm thm}. A linear operator $T\colon L^{p(\cdot)}(\Omega_1, \nu)\to L^{p(\cdot)}(\Omega_1, \nu)$ satisfies \eqref{MZ p variable} with $q(\cdot)=p(\cdot)$ and $r<p(\cdot)$ or $r>p(\cdot)$ $\nu$-a.e. if and only if for each $u\in L_+^{\alpha(\cdot)}(\Omega_1, \nu)$ there exists $w\in L_+^{\alpha(\cdot)}(\Omega_1, \nu)$ such that
\begin{itemize}
\item $u(x)\leq w(x)$ for $\nu$-a.e. $x\in \Omega_1$;
\item $\|w\|_{\alpha(\cdot)}\leq \max\{2, 2^{r/r'}\}\,\|u\|_{\alpha(\cdot)}$;
\item $T\colon L^r(w^{\sigma}d\nu)\to L^r(w^{\sigma}d\nu)$ is bounded, where $\sigma=1$ if $r<p(\cdot)$ and $\sigma=-1$ if $r>p(\cdot)$.
\end{itemize}
Moreover, the norm of $T\colon L^r(w^{\sigma}d\nu)\to L^r(w^{\sigma}d\nu)$ does not depend on the weight $u$.
\end{proposition}

Now we are ready to state our first application of Theorem~\ref{teoDefJunVariable}. In some sense, and following \cite[Chapter~VI.5, Corollary~5.4]{GarRub}, we show (for certain ranges of $p(\cdot), q(\cdot), r$) that the information concerning the boundedness of an operator $T\colon L^{q(\cdot)}(\Omega_2, \mu)\to L^{p(\cdot)}(\Omega_1, \nu)$ is contained in weighted $L^r$-inequalities.

\begin{corollary}\label{application weighted ineq}
We follow the notation in Proposition~\ref{weighted norm thm}.
Let $T\colon L^{q(\cdot)}(\Omega_2, \mu)\to L^{p(\cdot)}(\Omega_1, \nu)$ be any linear operator.
\begin{enumerate}[label=(\roman*)]
\item If $2\leq r<\min\{p_-, q_-\}$ then  $T\colon L^{q(\cdot)}(\Omega_2, \mu)\to L^{p(\cdot)}(\Omega_1, \nu)$ is bounded if and only if given $u\in L_+^{\alpha(\cdot)}(\Omega_1, \nu)$, there exist $U\in L_+^{\beta(\cdot)}(\Omega_2, \mu)$ such that $\|U\|_ {\beta(\cdot)}\leq \|u\|_ {\alpha(\cdot)}$ and $T\colon L^r(\Omega_2, U\,d\mu) \to L^r(\Omega_1, u\,d\nu)$ is bounded.
\item If $\max\{p_+, q_+\}<r\leq 2$ then $T\colon L^{q(\cdot)}(\Omega_2, \mu)\to L^{p(\cdot)}(\Omega_1, \nu)$ is bounded if and only if given $u\in L_+^{\beta(\cdot)}(\Omega_2, \mu)$, there exist $U\in L_+^{\alpha(\cdot)}(\Omega_1, \nu)$ such that $\|U\|_ {\alpha(\cdot)}\leq \|u\|_ {\beta(\cdot)}$ and $T\colon L^r(\Omega_2, u^{-1}\,d\mu) \to L^r(\Omega_1, U^{-1}\,d\nu)$ is bounded.
\item In the case $q(\cdot)=p(\cdot)$ we can unify the weights as follows:
if $2\leq r<p_-$ (respectively, $p_+<r\leq 2$), then $T\colon L^{p(\cdot)}(\Omega_1, \nu)\to L^{p(\cdot)}(\Omega_1, \nu)$ is bounded if and only if given $u\in L_+^{\alpha(\cdot)}(\Omega_1, \nu)$ there exist $w\in L_+^{\alpha(\cdot)}(\Omega_1, \nu)$ as in Proposition~\ref{coro peso unificado} such that $T\colon L^r(w\,d\nu)\to L^r(w\,d\nu)$ is bounded (respectively, $T\colon L^r(w^{-1}d\nu)\to L^r(w^{-1}d\nu)$ is bounded).
\end{enumerate}
\end{corollary}
We remark that there are situations  other than those stated in Corollary \ref{application weighted ineq} in which the characterization of the boundedness of $T$ in terms of weighted $L^r$-inequalities holds. Particularly, this happens in cases where the extreme values of $p(\cdot)$ or $q(\cdot)$ are not attained, for example if  $2<q_-<q(\cdot)<p(\cdot)$ almost everywhere and $2\le r\le \min\{p_-, q_-\}$.

\subsection{Vector-valued extensions of fractional operators}
Vector-valued inequalities for singular operators are a widely studied area in harmonic analysis. It is well-known that many singular classical operators (such as Calder\'on-Zygmund operators or fractional integral operators) have $\ell^r$-valued extensions (that is, equation \eqref{MZ} holds) for every $1\leq r<\infty$. These vector-valued extensions hold even in the weighted case, for weights in the classical Muckenhoupt classes, as an immediate consequence of extrapolation results (see, for instance, \cite{CruMarPer2} and references therein). In the variable exponent setting, despite being a much more recent research area, the study of boundedness and vector-valued extensions of singular operators is very rich. In \cite{CruWang}, Cruz-Uribe and Wang proved extrapolation results in variable Lebesgue spaces, including a generalization of the classical Rubio de Francia extrapolation theorem, off-diagonal and limited range extrapolation. As in the classical scenario, these extrapolations results give rise to $\ell^r$-valued inequalities for a wide class of singular operators between variable Lebesgue spaces. There exist, however, some singular operators which do not fit in this context. In \cite{RafSam}, the authors consider fractional integral operators of variable order $\alpha(\cdot)$ with rough kernels. These operators are defined by
$$
I_{\Omega}^{\alpha(\cdot)}f(x)=\int_{\mathbb{R}^n} \frac{\Omega(x-y)}{|x-y|^{n-\alpha(x)}}f(y)\,dy,
$$
where $\Omega\in L^1(\mathbb{S}^{n-1})$ is a homogeneous function of degree zero. In \cite[Theorem~4.4]{RafSam} it is proved that if $U\subset \mathbb{R}^n$ is a bounded open set, $\Omega\in L^s(\mathbb{S}^{n-1})$, $1<s<\infty$, $q(\cdot)\in \mathcal{P}_b(U)$ satisfy the $\log$-condition
\begin{equation}\label{log-condition}
|q(x)-q(y)|\leq \frac{A}{-\log\|x-y\|}, \quad \text{$\|x-y\|\leq \frac{1}{2}$, $x, y\in U$}
\end{equation}
for some constant $A>0$ depending only on $q(\cdot)$, and
$$
\inf_{x\in U}\alpha(x)>0, \quad \sup_{x\in U}\alpha(x)q(x)<n, \quad \text{and}\quad s\geq (q'(\cdot))_+,
$$
then the operator $I_\Omega^{\alpha(\cdot)}$ is bounded from $L^{q(\cdot)}(U)$ into $L^{p(\cdot)}(U)$, where $\frac{1}{p(\cdot)}=\frac{1}{q(\cdot)}-\frac{\alpha(\cdot)}{n}$. The approach in \cite{RafSam} is based on some pointwise estimates and do not use extrapolation results. This, precisely, allows to consider fractional operators of variable order $\alpha(\cdot)$, for which the off-diagonal extrapolation theorems do not work. Now, as a consequence of the result quoted above and our Theorem~\ref{teoDefJunVariable}, we obtain the following vector-valued inequalities for fractional operators of variable order.
\begin{corollary}\label{vector valued fractional}
Let $U\subset \mathbb{R}^n$ be a bounded open set, $\Omega\in L^s(\mathbb{S}^{n-1})$, $1<s<\infty$, a homogeneous function of degree zero and $q(\cdot)\in \mathcal{P}_b(U)$ satisfying the $\log$-condition \eqref{log-condition}.
If
$$
\inf_{x\in U}\alpha(x)>0, \quad \sup_{x\in U}\alpha(x)q(x)<n,\quad s\geq (q'(\cdot))_+ \quad \text{and}\quad \frac{1}{p(\cdot)}=\frac{1}{q(\cdot)}-\frac{\alpha(\cdot)}{n},
$$
then the operator $I_\Omega^{\alpha(\cdot)}$ satisfies the vector-valued inequalities
$$
 \left\| \left(\sum_{k=1}^N |I_\Omega^{\alpha(\cdot)}(f_k)|^r\right)^{1/r} \right\|_{L^{p(\cdot)}(U)}  \leq C \left\| \left(\sum_{k=1}^N |f_k|^r\right)^{1/r} \right\|_{L^{q(\cdot)}(U)}
$$
for every $r\in I(p_-, q_+)$, every $(f_k)_{k=1}^N\subset L^{q(\cdot)}(U)$ and every $N\in \mathbb{N}$.
\end{corollary}

\begin{remark}\rm
In \cite[Section 4]{Izu}, the author shows conditions ensuring the existence of vector-valued extensions of operators on Herz spaces.
Introduced by Herz \cite{Her} in his study of absolutely convergent Fourier transforms, these spaces are closely related to the study of functions and multipliers on classical Hardy spaces. By analogy with the definition of the classical Herz spaces, Izuki defines in \cite{Izu} the homogeneous and non-homogeneous Herz spaces with variable exponent, $\dot{K}_{p(\cdot)}^{\alpha, q}(\mathbb{R}^n)$ and $K_{p(\cdot)}^{\alpha, q}(\mathbb{R}^n)$. Combining our Theorem \ref{teoDefJunVariable} and \cite[Theorem 4.1]{Izu} we obtain the following result:
let $p(\cdot)\in \mathcal{P}_b(\mathbb R^n)$ be such that the Hardy-Littlewood maximal operator is bounded on $L^{p(\cdot)}(\mathbb{R}^n)$, $r\in I(p_-,p_+)$ and $\alpha\in (-n \delta_1, n \delta_2) $ (for certain $\delta_1, \delta_2$ specified in \cite[Section~3]{Izu}).
If $T$ is a bounded operator on $L^{p(\cdot)}(\mathbb R^n)$ satisfying the very general size condition
\begin{equation}\label{size condition}
|Tf(x)|\le C \int_{\mathbb R^n}|x-y|^{-n} |f(y)| dy
\end{equation}
for all $f\in L^{1}(\mathbb R^n)$, then $T$ admits a vector-valued extension on $K^{\alpha, q}_{p(\cdot)}(\mathbb R^n,\ell^r)$. In other words, there exists a constant $\widetilde C >0$ such that
\begin{equation*}
 \left\| \left(\sum_{k=1}^N |T(f_k)|^r\right)^{1/r} \right\|_{K^{\alpha, q}_{p(\cdot)}(\mathbb R^n)}  \leq \widetilde C \|T\| \left\| \left(\sum_{k=1}^N |f_k|^r\right)^{1/r} \right\|_{K^{\alpha, q}_{p(\cdot)}(\mathbb R^n)}.
\end{equation*}
It is worth noting that condition \eqref{size condition} is satisfied by many operators in harmonic analysis, such as Calder\'on-Zygmund operators, strongly singular multiplier operators and Bochner-Riesz means at the critical index (see, for instance, \cite{SorWei}).
\end{remark}

\section{Marcinkiewicz-Zygmund inequalities revisited: measure spaces with atoms}\label{Oliver Atom}
In Section~\ref{MaZy variable}, for $p(\cdot)\in \mathcal{P}_b(\Omega_1,\nu)$, $q(\cdot)\in \mathcal{P}_b(\Omega_2,\mu)$,  $\Omega_1$ and $\Omega_2$ non-atomic, we characterize  $1\leq r\leq \infty$ such that $k_{q(\cdot), p(\cdot)}(r)<\infty$  in terms of $p_-$ and $q_+$. As we will see below, if the underlying measure spaces have atoms, the range of $r$'s such that $k_{q(\cdot), p(\cdot)}(r)<\infty$ can be {larger} than in the non-atomic case. This is not an unexpected result: consider a measure space $(\Omega_2, \mu)$ which has an atom $\{\omega_0\}$ and such that $\Omega_2\setminus \{\omega_0\}$ is atomless.  If $q\colon \Omega_2 \to [1, \infty]$ is such that $q(\omega_0) = 2$ and
$$\esssup_{x\in \Omega\setminus \{\omega_0\}}q(x) =\frac32,$$
it is reasonable to expect that \emph{every} operator $T\colon L^{q(\cdot)}(\Omega_2, \mu)\to L^{2}(\mathbb{R}, dx)$ has bounded $\ell^r$-valued extension if and only if $r\in I(2,  3 / 2)=[3/2,2) \supsetneq \{2\}= I(2, q_+)$. This is, indeed, the case, since the boundedness of $T\colon L^{q(\cdot)}(\Omega_2, \mu)\to L^{2}(\mathbb{R}, dx)$ and its vector-valued extensions do not depend on the behavior of the exponents on finitely many atoms.
With this in mind, given exponents $p(\cdot)\in \mathcal{P}_b(\Omega_1,\nu)$, $q(\cdot)\in \mathcal{P}_b(\Omega_2,\mu)$ we define
$$
\widetilde{p}_-=\sup\{t:\,\, \{p(\cdot)<t\}\,\, \text{is finite (or null) union of atoms}\}
$$
and
$$
\widetilde{q}_+=\inf\{t:\,\, \{q(\cdot)>t\}\,\, \text{is finite (or null) union of atoms}\}.
$$
It is easy to check that $p_-\leq \widetilde{p}_-$ and $\widetilde{q}_+\leq q_+$.

The existence of $\ell^r$-valued extensions of an operator $T\colon  L^{q(\cdot)}(\Omega_2, \mu)\to L^{p(\cdot)}(\Omega_1, \nu)$ in the case of measures with atoms follows from the corresponding extensions of some restrictions of $T$. In this direction, we state the following lemma which is easy to verify.

\begin{lemma}\label{extension sum operators}
Suppose that the operators $T_1, T_2\colon  L^{q(\cdot)}(\Omega_2, \mu)\to L^{p(\cdot)}(\Omega_1, \nu)$ admit bounded $\ell^r$-valued extensions $\widetilde{T}_1$ and $\widetilde{T}_2$ with norms $\|\widetilde{T}_1\|\leq C\|T_1\|$ and $\|\widetilde{T}_2\|\leq C\|T_2\|$ for some constant $C\geq 1$. Then $T_1+T_2$ admit a bounded $\ell^r$-valued extension $\widetilde{T_1+T_2}$ with norm $\|\widetilde{T_1+T_2}\|\leq C\left(\|T_1\|+\|T_2\|\right)$.
\end{lemma}

Now we are ready to state and prove the main result of this section.

\begin{theorem}\label{teoDefJunAtomic}
Let $(\Omega_1,\nu)$ and $(\Omega_2,\mu)$ be measure spaces which are not union of finitely many atoms and let $p(\cdot)\in \mathcal{P}_b(\Omega_1,\nu)$, $q(\cdot)\in \mathcal{P}_b(\Omega_2,\mu)$ and $1\leq r\leq\infty$. If $k_{q(\cdot), p(\cdot)}(r)<\infty$ then $r \in I(\widetilde{p}_-,\widetilde{q}_+)$. Conversely,  $r \in {\rm int}({I}(\widetilde{p}_-,\widetilde{q}_+))\cup \{2\}$, then $k_{q(\cdot), p(\cdot)}(r)<\infty$.
\end{theorem}
As we will see in Remark~\ref{ejemplos borde}, both implications in Theorem~\ref{teoDefJunAtomic} are sharp. As in the proof of Theorem~\ref{teoDefJunVariable} we use $\lesssim$ to mean that an inequality holds up to a universal constant.

\begin{proof}[Proof of Theorem~\ref{teoDefJunAtomic}]
Suppose first that $k_{q(\cdot), p(\cdot)}(r)<\infty$ and let $\epsilon>0$. Given $\epsilon>0$ consider the subsets
\begin{equation*}
		\Omega_1^{\widetilde{p}_-+\epsilon}=\lbrace x\in\Omega_1:\,\,p(x)<\widetilde{p}_-+\epsilon\rbrace
	\quad
	\text{and}
	\quad
			\Omega_2^{\widetilde{q}_+-\epsilon}=\lbrace x\in\Omega_2:\,\,q(x)>\widetilde{q}_+-\epsilon\rbrace
	\end{equation*}
(we follow the notation in the proof of Theorem~\ref{teoDefJunVariable}, where $\Omega_1^{p_-+\epsilon}$ and $\Omega_2^{q_+-\epsilon}$ were defined analogously).	
Since $p_-\leq \widetilde{p}_-$ and $\widetilde{q}_+\leq q_+$ we have $\Omega_1^{p_-+\epsilon}\subseteq \Omega_1^{\widetilde{p}_-+\epsilon}$ and $\Omega_2^{q_+-\epsilon}\subseteq \Omega_2^{\widetilde{q}_+-\epsilon}$ and, consequently, $\nu(\Omega_1^{\widetilde{p}_-+\epsilon})>0$ and $\mu(\Omega_2^{\widetilde{q}_+-\epsilon})>0$.
Note also that $(\Omega_1^{\widetilde{p}_-+\epsilon}, \nu)$ and $(\Omega_2^{\widetilde{q}_+-\epsilon}, \mu)$ are not union of finitely many atoms (by definition of $\widetilde{p}_-$ and $\widetilde{q}_+$). By Lemma~\ref{lema 2}, monotonicity and Lemma~\ref{lema 1} we have
\begin{eqnarray*}
\nonumber k_{\widetilde{q}_+-\epsilon, \widetilde{p}_-+\epsilon}(r)&=&k_{L^{\widetilde{q}_+-\epsilon}(\Omega_2^{\widetilde{q}_+-\epsilon}), L^{\widetilde{p}_-+\epsilon}(\Omega_1^{\widetilde{p}_-+\epsilon})}(r)\\
\nonumber &\lesssim& k_{L^{q(\cdot)}(\Omega_2^{\widetilde{q}_+-\epsilon}), L^{p(\cdot)}(\Omega_1^{\widetilde{p}_-+\epsilon})}(r)\\
&\lesssim&  k_{L^{q(\cdot)}(\Omega_2), L^{p(\cdot)}(\Omega_1)}(r)<\infty
\end{eqnarray*}
and, since $\epsilon>0$ was arbitrary,
$$
r\in\bigcap_{\epsilon>0}I(\widetilde{p}_-+\epsilon,\widetilde{q}_+-\epsilon).
$$
Then, reasoning as in the proof Theorem~\ref{teoDefJunVariable} we obtain $r \in I(\widetilde{p}_-,\widetilde{q}_+)$.

Suppose now that $r \in {\rm int}({I}(\widetilde{p}_-,\widetilde{q}_+))\cup \{2\}$. Then we have $r \in I(\widetilde{p}_--\epsilon,\widetilde{q}_++\epsilon)$ for some $\epsilon>0$. For this fixed $\epsilon>0$ consider
$$
\Theta_1=\{x\in \Omega_1:\,\, p(x)>\widetilde{p}_--\epsilon\} \quad \text{and}\quad \Theta_2=\{x\in \Omega_2:\,\, q(x)<\widetilde{q}_+ +\epsilon\}
$$
and note that, by definition of $\widetilde{p}_-$ and $\widetilde{q}_+$, their complements $\Theta_1^c$ and $\Theta_2^c$ are finite (or null) unions of atoms. Consequently, $\Theta_1$ and $\Theta_2$ are not union of finitely many atoms. Now, given $T\colon L^{q(\cdot)}(\Omega_2, \mu)\to L^{p(\cdot)}(\Omega_1, \nu)$ we can write
$$
T=T_1+T_2+T_3,
$$
where $T_i\colon L^{q(\cdot)}(\Omega_2, \mu)\to L^{p(\cdot)}(\Omega_1, \nu)$, $i=1,2,3$, are given by
$$
T_1(f)= T\left(f\cdot \chi_{\Theta_2}\right)\cdot \chi_{\Theta_1}, \quad T_2(f)= T\left(f\cdot \chi_{\Theta_2}\right)\cdot \chi_{\Theta_1^c}\quad \text{and}\quad T_3(f)=T\left(f\cdot \chi_{\Theta_2^c}\right).
$$
Let us prove that $T_1, T_2$ and $T_3$ have bounded $\ell^r$-valued extension. First, since $r \in I(\widetilde{p}_--\epsilon,\widetilde{q}_++\epsilon)$, by monotonicity and Lemma~\ref{lema 2}   we have
$$
k_{L^{q(\cdot)}(\Theta_2), L^{p(\cdot)}(\Theta_1)}(r)\lesssim  k_{L^{\widetilde{q}_++\epsilon}(\Theta_2), L^{\widetilde{p}_--\epsilon}(\Theta_1)}(r)= k_{\widetilde{q}_++\epsilon, \widetilde{p}_--\epsilon}(r) <\infty
.$$
Note also that
$T_1\vert_{L^{q(\cdot)}(\Theta_2, \mu)} $ continuously maps $L^{q(\cdot)}(\Theta_2, \mu)$ in $ L^{p(\cdot)}(\Theta_1, \nu)$.
Then, given $(f_k)_{k=1}^N\subset L^{q(\cdot)}(\Omega_2, \mu)$ and $r \in I(\widetilde{p}_--\epsilon,\widetilde{q}_++\epsilon)$ we have
\begin{eqnarray*}
 \left\| \left(\sum_{k=1}^N |T_1(f_k)|^r\right)^{1/r} \right\|_{L^{p(\cdot)}(\Omega_1, \nu)}&=& \left\| \left(\sum_{k=1}^N |T_1\vert_{L^{q(\cdot)}(\Theta_2, \mu)}(f_k\cdot \chi_{\Theta_2})|^r\right)^{1/r} \right\|_{L^{p(\cdot)}(\Theta_1, \nu)}\\
&\lesssim&  k_{\widetilde{q}_++\epsilon, \widetilde{p}_--\epsilon}(r) \|T_1\|
\left\| \left(\sum_{k=1}^N |f_k\cdot \chi_{\Theta_2}|^r\right)^{1/r} \right\|_{L^{q(\cdot)}(\Theta_2, \nu)}\\
&\lesssim& k_{\widetilde{q}_++\epsilon, \widetilde{p}_--\epsilon}(r) \|T_1\|
\left\| \left(\sum_{k=1}^N |f_k|^r\right)^{1/r} \right\|_{L^{q(\cdot)}(\Omega_{2}, \nu)}.
\end{eqnarray*}
 This proves that, if $r \in {\rm int}({I}(\widetilde{p}_-,\widetilde{q}_+))\cup \{2\}$, then $T_1$ has bounded $\ell^r$-valued extension $\widetilde{T}_1$ with norm $\|\widetilde{T}_1\|\lesssim k_{\widetilde{q}_++\epsilon, \widetilde{p}_--\epsilon}(r) \|T_1\|$.

 On the other hand,  $T_2$ and $T_3$ are, essentially,  operators with finite-dimensional co-domain and finite-dimensional domain, respectively. Then,  $T_2$ and $ T_3$ clearly have bounded $\ell^r$-valued extension $\widetilde{T}_1$ and $\widetilde{T}_2$. Therefore,  $\widetilde T=\widetilde T_1+\widetilde T_2+\widetilde T_3$ is a bounded $\ell^r$-valued extension of $T$. Since $T$ was arbitrary, we have
$
k_{q(\cdot), p(\cdot)}(r) <\infty
$
for every $r \in {\rm int}({I}(\widetilde{p}_-,\widetilde{q}_+))\cup \{2\}$.
\end{proof}

\begin{remark}\label{ejemplos borde}\rm
Let us see that both implications in the statement in Theorem~\ref{teoDefJunAtomic} are sharp. On the one hand, for $(\Omega_1,\nu)$ and $(\Omega_2,\mu)$ non-atomic measure spaces, we  know (by Theorem~\ref{teoDefJunVariable}) that if $k_{q(\cdot), p(\cdot)}(r)<\infty$ then  $r \in I(\widetilde{p}_-,\widetilde{q}_+)=I(p_-, q_+)$.
On the other hand, to see that the second implication is also sharp, consider
$\Omega_1=\Omega_2=\mathbb{N}$ with the counting measure $\mu$. We choose $1<q_0<p_0<2$ and take  $(\lambda_n)_n$ strictly decreasing to $0$ such that
$n^{\lambda_n}\to \infty$. Let $p(\cdot)\in \mathcal{P}_b(\mathbb{N},\mu)$ be the constant exponent  $p(n)= p_0$ for every $n\in \mathbb{N}$ and define $q(\cdot)\in \mathcal{P}_b(\mathbb{N},\mu)$  by
$
q(n)=q_0+\lambda_n
$.
Note that $\widetilde{q}_+=q_0$, and let us show that
$k_{q(\cdot), p_0}(q_0)=\infty$. First note, by Lemma~\ref{lema 1} and monotonicity, that
\begin{equation}\label{contraejemplo}
k_{q(\cdot), p_0}(q_0)\geq k_{L^{q(\cdot)}(\{1,\dots, n\}, \mu), \ell_{p_0}^n}(q_0)\gtrsim k_{\ell_{q_0+\lambda_n}^n, \ell_{p_0}^n}(q_0),
\end{equation}
where we write $\ell_p^n$ for $(\mathbb{R}^n, \|\cdot\|_p)$. Now, by \cite[Section 6]{DefJun} we know that
$$
k_{\ell_{q_0+\lambda_n}^n, \ell_{p_0}^n}(q_0)\asymp n^{\frac{1}{q_0}-\frac{1}{q_0+\lambda_n}}
$$
and, since  $n^{\lambda_n}\to \infty$, we have $k_{\ell_{q_0+\lambda_n}^n, \ell_{p_0}^n}(q_0)\to \infty$. This and \eqref{contraejemplo} gives that $k_{q(\cdot), p_0}(q_0)=\infty$.
\end{remark}

\noindent \textbf{Acknowledgements.} We want to thank our friend Sheldy Ombrosi for useful conversations and comments regarding this work. We also thank the anonymous referees for many suggestions and comments that helped to improve the manuscript. Finally, we gratefully acknowledge the support provided by the scientific system of Argentina. Furthermore, we wish to express our concern regarding the ongoing defunding that universities and scientific agencies are currently facing.

\noindent \textbf{Disclosure statement.} The authors declare that they have no conflicts of interest.

\end{document}